\documentclass[a4paper]{scrartcl}
\usepackage{amsmath}
\usepackage{amssymb}
\usepackage{enumerate}
\usepackage{amsthm}
\usepackage[obeyDraft]{todonotes}
\usepackage[UKenglish]{babel}
\usepackage[T1]{fontenc}
\usepackage[utf8]{inputenc}
\usepackage[hyperfootnotes=false]{hyperref}

\newtheorem{theorem}{Theorem}[section]
\newtheorem{lemma}[theorem]{Lemma}
 
\theoremstyle{definition}

\theoremstyle{remark}
\newtheorem{remark}[theorem]{Remark}
\numberwithin{equation}{section}

\newcommand{\R}{\mathbb{R}}
\newcommand{\Q}{\mathbb{Q}}

\newcommand{\lip}{\operatorname{Lip}}

\newcommand{\M}{\mathcal{M}}
\newcommand{\N}{\mathcal{N}}
\DeclareMathOperator{\diam}{diam}

\newcommand\blfootnote[1]{%
  \begingroup
  \renewcommand\thefootnote{}\footnote{#1}%
  \addtocounter{footnote}{-1}%
  \endgroup
}

\title{$\sigma$-Porosity of the set of strict~contractions in a space of non-expansive~mappings.}
\author{Christian Bargetz \and Michael Dymond}
\date{1st September 2015}

\begin{document}
\maketitle
\begin{abstract}
  \textbf{Abstract.}
  We consider the space of non-expansive mappings on a bounded, closed and convex subset of a
  Banach space equipped with the metric of uniform convergence. We show that the set of strict
  contractions is a \linebreak$\sigma$-porous subset. If the underlying Banach space is separable, we exhibit a $\sigma$-porous subset of the space of non-expansive~mappings outside of which all mappings attain the maximal Lipschitz constant one at typical points of their domain. \blfootnote{The final publication is available via \url{http://dx.doi.org/10.1007/s11856-016-1372-z}}
\end{abstract}
\section{Introduction}

In the context of fixed point theory, the space of non-expansive mappings $f\colon C\to C$ 
on a closed, convex and bounded subset $C$ of a Banach space $X$ has been well studied. 
F.~S.~De~Blasi and J.~Myjak came upon the question of the size of the set of strict contractions in 
this space. In~\cite{dBM89:Porosity}, they prove that if $X$ is a Hilbert 
space the set of strict contractions on $C$ is negligible in the sense that it is $\sigma$-porous. 
In~\cite{Rei05:GenericityPorosity} S.~Reich formulated the question of whether this is 
also true in general Banach spaces. An important tool in the proof of De Blasi and Myjak, 
contained in~\cite{dBM76:Convergence}, is Kirszbraun's extension theorem for 
Lipschitz mappings which is not available for general Banach spaces. Therefore it is not 
possible to extend this proof to Banach spaces. The aim of this article is to answer the
question of S.~Reich in the positive, i.e., proving $\sigma$-porosity of the set of strict
contractions in the Banach space setting. In the case of separable Banach spaces we get the
stronger result that outside a $\sigma$-porous subset of the space of non-expansive mappings, all non-expansive mappings have 
the maximal possible Lipschitz constant one at typical points of their domain.

In~\cite{Rak62:Contractive} the concept of contractive mappings is introduced. A non-expansive mapping 
$f\colon C\to C$ is called contractive, if it satisfies
\[
\|f(x)-f(y)\| \leq \phi^f(\|x-y\|)\|x-y\|
\]
for a decreasing function $\phi^f\colon [0,\operatorname{diam}(C)]\to[0,1]$ with $\phi^f(t)<1$ for
$t>0$. This means that on large scales $f$ behaves like a strict contraction but on small scales
it can approximate an isometry.
Therefore being a strict contraction is a stronger assumption than being a contractive mapping.
S.~Reich and A.~J.~Zaslavski show in~\cite{RZ01:Noncontractive} that the set of 
non-contractive mappings on $C$ is $\sigma$-porous. Combining the results of the present paper 
with those of~\cite{RZ01:Noncontractive}, we see that, except for a $\sigma$-porous set, all 
non-expansive mappings are contractive with Lipschitz constant one.
In~\cite{dBMRZ09:GenericExistence} and~\cite{PL14:PorosityFixedPoints} similar problems for 
set-valued~mappings are considered.

A detailed discussion of the fixed point problem for non-expansive mappings can be found in 
Chapter~3 of the book~\cite{BL00:GeometricNonlinear}.

Given a non-empty metric space $M$ with metric $d$, a point $x\in M$ and $r>0$ we denote by $B(x,r)$
the open ball with centre $x$ and radius $r$. A subset $A\subset M$ is said to be 
\emph{porous at $x\in A$} if there are constants $\alpha>0$ and $\varepsilon_0>0$ with the following 
property: For all $\varepsilon\in (0,\varepsilon_0)$ there is a point $y\in M$ with 
$\left\|y-x\right\|\leq \varepsilon$ and $B(y, \alpha\,\varepsilon)\cap A=\emptyset$. The set $A$ is 
called \emph{porous} if it is porous at all of its points. Some authors call sets satisfying this strong porosity condition \emph{lower porous}, see for example~\cite{Zaj05:OnSigmaPorous}.
A set $F\subset M$ is called \emph{$\sigma$-porous} if $F$ may be expressed as a countable union of porous sets. Note that 
every $\sigma$-porous set is a set of first category in the sense of the Baire category theorem.
For a detailed overview of different notions of porosity, we refer to~\cite{Zaj05:OnSigmaPorous}.

\section{Construction}
Let $X$ be a Banach space with norm $\|\cdot\|$ and unit sphere $S(X)$. Given two elements $x,y\in X$, we will write $[x,y]$ for the line segment with endpoints $x$ and $y$. Let $C\subset X$ be a bounded, closed and convex subset with more than one element. 
For a mapping $f:C\to C$, we define the Lipschitz constant of $f$ by
\begin{equation*}
\lip(f)=\sup\left\{\frac{\left\|f(y)-f(z)\right\|}{\left\|y-z\right\|}\quad\colon\quad y,z\in C,\quad y\neq z\right\}.
\end{equation*}
Given a vector $x\in C$ and a set $U \subseteq C$, we further define the Lipschitz constant of $f$ at $x$ and on $U$ respectively by
\begin{align*}
\lip(f,x)&:= \limsup_{r\to 0+} \left\{\frac{\left\|f(y)-f(x)\right\|}{\left\|y-x\right\|}\quad\colon\quad y\in B(x,r)\cap C,\quad y\neq x\right\},
\end{align*}
and
\[
\lip(f,U):=\lip(f\arrowvert_{U})
=\sup\left\{\frac{\left\|f(y)-f(z)\right\|}{\left\|y-z\right\|}\quad\colon\quad y,z\in U,\quad y\neq z\right\}.
\]
A mapping $f\colon C\to C$ is called \emph{non-expansive} if $\lip(f)\leq 1$; $f$ is called a \emph{strict contraction} if $\lip(f)<1$.
We denote by $\mathcal{M}$ the set of all non-expansive mappings $f\colon C\to C$ and write $\N$ for the subset of $\M$ formed by the strict contractions. Given a non-empty, open, convex subset $U$ of $C$, we will also consider the set
\begin{equation*}
\N(U)=\left\{f\in\M\mbox{ : }\lip(f,U)<1\right\}.
\end{equation*}
of all mappings in $\M$ whose Lipschitz constant on $U$ is strictly less than one. Note that $\N(U)$ contains the set $\N$ of strict contractions.

The set $\M$ together with the metric of uniform convergence $d$, given by 
\[
d(f,g)=\left\|f-g\right\|_{\infty}=\sup_{x\in C}\|f(x)-g(x)\|,
\]
form a complete metric space. 
Let us now state our main results.

\begin{theorem}\label{theorem:mainresult}
  The set $\N$ of all strict contractions is a $\sigma$-porous subset of $\mathcal{M}$.
\end{theorem}

If $X$ is a separable Banach space we get the following stronger result:

\begin{theorem}\label{theorem:mainresult2}
  Let $X$ be a separable Banach space.
  Then there is a $\sigma$-porous set $\widetilde{\N}\subset M$ such that for every $f\in\M\setminus \widetilde{\N}$, 
  the set 
  \begin{equation}\label{eq:Rf}
    R(f)=\left\{x\in C\quad\colon\quad \lip(f,x)=1\right\}
  \end{equation}
  is a residual subset of $C$.
\end{theorem}
Put differently, Theorem~\ref{theorem:mainresult2} says that outside of a negligible subset of $\M$, all mappings
in the space $\M$ attain the maximal possible Lipschitz constant $1$ at typical points of their 
domain $C$. Note that the conclusion $\lip(f,x)=1$ on a residual set is stronger than the conclusion
$\lip(f)=1$. 

Let us begin working towards the proofs of Theorems~\ref{theorem:mainresult} 
and~\ref{theorem:mainresult2}:

Fix a non-empty, open, convex subset $U$ of $C$. We will prove that $\N(U)$ is a $\sigma$-porous subset of $\M$. For $a,b\in(0,1)$, let 
\begin{equation*}
  \N_{a}^{b}(U)=\left\{f\in\N(U)\colon a<\lip(f,U)\leq b\right\}.
\end{equation*}
The significance of the above decomposition of $\N(U)$ is revealed in the following lemma.
\begin{lemma}\label{lemma:porous}
  If $a,b\in(0,1)$ satisfy the condition
  \begin{equation}\label{eq:b-a}
  b-a <\min\left\{\frac{a}{16}, \frac{a(1-b)}{48b}\right\},
  \end{equation}
  then the set $\N_{a}^{b}(U)$ is porous.
\end{lemma}
The discussion and lemmata which follow form the basis of the proof of Lemma~\ref{lemma:porous}. We will need the following known property of Lipschitz mappings on a convex set:
\begin{lemma}\label{lemma:x0}
	Let $K$ be a convex subset of $X$, $f\colon K\to X$ be a Lipschitz mapping and $0<L<\lip(f)$. Then there exist $x_{0}\in K$ and $e\in S(X)$ such that
	\begin{equation}\label{eq:x_0}
	\liminf_{t\to 0+}\frac{\left\|f(x_{0}+te)-f(x_{0})\right\|}{t}>L.
	\end{equation}

\end{lemma}
\begin{proof}
	Choose $x,y\in K$ such that
		\begin{equation*}
		\frac{\left\|f(y)-f(x)\right\|}{\left\|y-x\right\|}>L.
		\end{equation*}
		Let $e=(y-x)/\left\|y-x\right\|\in S(X)$. Assume for a contradiction that for every $x_{0}\in [x,y]\setminus\left\{y\right\}$ the limit inferior given in \eqref{eq:x_0} is at most $L$. Let $\rho$ be a norm one, linear functional on $X$ with the property $\rho(f(y)-f(x))=\left\|f(y)-f(x)\right\|$. Such a functional can be 
		constructed using the Hahn-Banach theorem. Then the restriction of $\rho\circ f$ to $[x,y]$, which we again denote by $\rho\circ f$, is a Lipschitz mapping from the interval $[x,y]$ to $\R$ and is therefore differentiable almost everywhere. Moreover, our assumption implies $\left\|(\rho\circ f)'\right\|_{\infty}\leq L$ and we obtain
		\begin{equation*}
		\left\|f(y)-f(x)\right\|=\rho(f(y))-\rho(f(x))=\int_{x}^{y}(\rho\circ f)'(t)\,{\rm d}t\leq L\left\|y-x\right\|.
		\end{equation*}
		This contradicts the choice of $x$ and $y$.
\end{proof}

Fix $0<a<b<1$, $f\in\N_{a}^{b}(U)$ and let the point $x_{0}\in U$ and the direction $e\in S(X)$ satisfying \eqref{eq:x_0} be given by the conclusion of Lemma~\ref{lemma:x0} when we take $K= U$, $L=a$ and replace $f$ with $f|_{U}$. Choose $r>0$ sufficiently small so that $B(x_{0},r)\cap C\subseteq U$, $x_{0}+re\in U$ and 
\begin{equation}\label{eq:stretch}
\frac{\left\|f(x_{0}+te)-f(x_{0})\right\|}{t}>a
\end{equation}
for all $t\in(0,r)$. Let $\sigma\in(0,1)$ be a constant satisfying $b(1+3\sigma)\leq 1$ which will be determined later in the proof and set $\varepsilon_{0}=\sigma r/2$. Fix $\varepsilon\in (0,\varepsilon_{0})$. We will prove Lemma~\ref{lemma:porous} by finding a mapping $g\in\M$ and a 
constant $\alpha$ depending only on $a$, $b$ with the properties $\left\|g-f\right\|_{\infty}\leq 
\varepsilon$ and $B(g,\alpha \varepsilon)\cap\N_{a}^{b}(U)=\emptyset$.

Define a $1$-Lipschitz function $\phi\colon\R\to\R$ by $\phi(t)=\min\left\{\left|t\right|,
  \varepsilon/\sigma\right\}$. 
Let $e^*$ denote a continuous linear functional on $X$ with $e^*(e) = 1$. Additionally let $\psi\colon X\to [0,1]$ be defined by
\[
\psi(x) = \begin{cases} 1-\frac{2}{r}\operatorname{dist}\left(x,B(0,r/2)\right) & x\in B(0,r) \\ 0 &x\not\in B(0,r)\end{cases}.
\]
The function $\psi$ is equal to one on the ball $B(0,r/2)$ and is $2/r$-Lipschitz. We define a mapping 
$g\colon C\to C$ by
\begin{equation}\label{eq:geps}
  g(x)=f\left(x+\frac{\sigma}{2r}\psi(x-x_{0})\phi(e^{*}(x-x_{0}))(re-(x-x_{0}))\right).
\end{equation}
Let us first verify that $g$ is indeed a mapping from $C$ to $C$. Fix $x\in C$. Then
\begin{equation}\label{eq:convexcomb}
  x+\frac{\sigma}{2r}\psi(x-x_{0})\phi(e^{*}(x-x_{0}))(re-(x-x_{0}))=(1-\lambda)x+\lambda(x_{0}+re)
\end{equation}
with $\lambda:=\frac{\sigma}{2r}\psi(x-x_{0})\phi(e^{*}(x-x_{0}))\in(0,1)$. As $x_0+re\in U\subseteq C$, by the choice of $r$, the expression above defines an element 
of $C$ in the form of a convex combination of elements of $C$. It is now clear that $g$ is a 
well-defined mapping from $C$ to $C$. 

Note that $g(x)=f(x)$ for $x\in C\setminus B(x_{0},r)$, since $\psi(x-x_{0})=0$. Hence, in particular, $\psi$ ensures that we do not increase the Lipschitz constant outside of the set $U$ when we go from $f$ to $g$. This is important because at points outside of $U$, $f$ may already have the maximal permissible Lipschitz constant one. In the case $U=C$, which suffices for the proof of Theorem~\ref{theorem:mainresult}, we may simply set $\psi=1$.

Next we will establish that $g$ is non-expansive. To this end, we will study separately the mapping $\gamma:X\to X$ defined by
\begin{align*}
  \gamma(x)=\frac{\sigma}{2r}\psi(x)\phi(e^{*}(x))(re-x).
\end{align*}
The mapping $\gamma$ can be thought of as the perturbation of the set $C$ around $x_{0}$ through which we obtain $g$ from $f$, because $g(x)=f(x+\gamma(x-x_{0}))$.
\begin{lemma}\label{lemma:gamma}
  The mapping $\gamma$ satisfies $\lip(\gamma)
    \leq3\sigma$.

\end{lemma}
  \begin{proof}
Using the formula $\lip(vw)\leq\lip(v)\left\|w\right\|_{\infty}+\lip(w)\left\|v\right\|_{\infty}$
twice we obtain 
\[\lip(\gamma)\leq\frac{\sigma}{2r}\left(\frac{2}{r}\frac{\varepsilon}{\sigma}2r+2r+\frac{\varepsilon}{\sigma}\right)\leq 3\sigma,\]
where we apply the following: $\lip(\psi)\leq2/r$, $\lip(\phi),\lip(e^{*})\leq1$, $\left\|\phi\right\|_{\infty}\leq\varepsilon/\sigma$,  $\left\|\psi\right\|_{\infty}\leq 1$, $\gamma(x)=0$ for $x\notin B(0,r)$ and finally the inequality $\varepsilon<\sigma r/2$.
  \end{proof}
  
  \begin{lemma}\label{lemma:lipg}
  	The mapping $g$ has the following properties:
  	\begin{align}
  	&\left\|g-f\right\|_{\infty}\leq \varepsilon.\label{1}\\
  	&\lip(g)\leq\max\left\{b\left(1+3\sigma \right),\quad\lip(f)\right\}\leq 1.\label{2}\\
  	&\frac{\left\|g(x_{0}+\varepsilon e)-g(x_{0})\right\|}{\varepsilon}\geq a\left(1+\frac{\sigma}{4}\right).\label{3}
  \end{align}
  	\end{lemma}
  \begin{proof}
  	Given $x\in C$, the observation
  	\begin{align*}        	\left\|g(x)-f(x)\right\|&=\left\|f\left(x+\gamma(x-x_{0})\right)-f(x)\right\|
  	\leq\lip(f) \left\|\gamma\right\|_{\infty} \leq \varepsilon
  	\end{align*}
  	verifies \eqref{1}.
  	
  	In order to prove \eqref{2}, it suffices to show that both $\lip(g,C\setminus B(x_{0},r))$ and $\lip(g,B(x_{0},r)\cap C)$ satisfy this bound.
  	
  	First consider two points $x,y\in C\setminus B(x_{0},r)$. Then $\psi(x-x_{0})=\psi(y-x_{0})=0$ and so $\gamma(x-x_{0})=\gamma(y-x_{0})=0$. Hence $g(x)=f(x)$, $g(y)=f(y)$ and we get $\left\|g(x)-g(y)\right\|\leq\lip(f)\left\|x-y\right\|$. Thus $\lip(f,C\setminus B(x_{0},r))\leq\lip(f)$.
  	
  	Secondly, let $x,y\in B(x_{0},r)\cap C$. For each $z\in U$ the expression $z+\gamma(z-x_{0})$ defines a convex combination of elements of $U$, as seen in \eqref{eq:convexcomb}. The convexity of $U$ now guarantees that $z+\gamma(z-x_{0})\in U$. Applying this argument with $z=x$ and $z=y$, we can use the estimate for the Lipschitz constant of $f$ restricted to $U$, together with Lemma~\ref{lemma:gamma}, to obtain the following bound:
  	\begin{align*}
  	\left\|g(x)-g(y)\right\|&=\left\|f(x+\gamma(x-x_{0}))-f(y+\gamma(y-x_{0}))\right\|\\
  	&\leq\lip(f,U)\left(1+\lip(\gamma)\right)\left\|x-y\right\|\\
  	&\leq b\left(1+3\sigma \right)\left\|x-y\right\|.
  	\end{align*}
  	This completes the proof of \eqref{2}.
  	
  	Finally let us show \eqref{3}: 	Using the definition of $g$, we observe that
  	\begin{align*}
  	\frac{\|g(x_0+\varepsilon e)-g(x_0)\|}{\varepsilon} &= \frac{\|f(x_0+\varepsilon e + \frac{\sigma\varepsilon}{2r}(r-\varepsilon)e)-f(x_0)\|}{\varepsilon+\frac{\sigma\varepsilon}{2r}(r-\varepsilon)}\times\frac{\varepsilon+\frac{\sigma\varepsilon}{2r}(r-\varepsilon)}{\varepsilon}\\
  	&\geq a \left(1+\frac{\sigma}{2r}(r-\varepsilon)\right)
  	\geq a\left(1+\frac{\sigma}{4}\right)
  	\end{align*}
  	using $\varepsilon<\sigma r/2$, $\sigma\in (0,1)$ and $\varepsilon+\frac{\sigma\varepsilon}{2r}(r-\varepsilon)<r$.
   \end{proof}

  \begin{lemma}\label{lemma:hole}
    Let $h:C\to C$ be a Lipschitz mapping with $\left\|h-g\right\|_{\infty}\leq\frac{a\sigma}{16} \varepsilon$. Then, 
	\[
	\lip(h,U)>a\left(1+\frac{\sigma}{8}\right).
	\]
  \end{lemma}
  \begin{proof}
    From \eqref{3}, we deduce, using $[x_{0},x_{0}+\varepsilon e]\subseteq B(x_{0},r)\cap C\subseteq U$, that
    \begin{align*}
      \lip(h,U)\geq \frac{\left\|h(x_{0}+\varepsilon e)-h(x_{0})\right\|}{\varepsilon}
      &\geq\frac{\left\|g(x_{0}+\varepsilon e)-g(x_{0})\right\|}{\varepsilon}-\frac{a\sigma}{8}\\
      &\geq a\left(1+\frac{\sigma}{4}\right)-\frac{a\sigma}{8}
      = a\left(1+\frac{\sigma}{8}\right).
    \end{align*}
  \end{proof}

We are now ready to present a proof of Lemma~\ref{lemma:porous}.
\begin{proof}[Proof of Lemma~\ref{lemma:porous}]
   Let $a,b\in(0,1)$ be given by the statement of Lemma~\ref{lemma:porous} and set
  \[
  \sigma = \frac{16(b-a)}{a}.
  \]
  Fix $f\in\N_{a}^{b}$ and let $\varepsilon_{0}>0$ be defined according to the discussion following \eqref{eq:stretch}. Given $\varepsilon\in(0,\varepsilon_{0})$, let $g$ be given by \eqref{eq:geps} in the above construction. Lemma~\ref{lemma:lipg} indicates that $g\in\M$ and $\left\|g-f\right\|_{\infty}\leq \varepsilon$. Moreover, Lemma~\ref{lemma:hole} asserts that the intersection $B(g,\alpha \varepsilon)\cap\N_{a}^{b}(U)=\emptyset$ for $\alpha =a\sigma/16=b-a$.
  \end{proof}
\begin{lemma}\label{lemma:constants}
  The set $\N_{0}(U)$ of all non-expansive mappings $f:C\to C$ which are constant on $U$ is a porous subset of $\M$.
\end{lemma}
\begin{proof}
  Fix two points $u,v\in U$ and set $R=\diam(C)$. Given $\varepsilon>0$ and a function $f\colon C\to C$ which is constant on $U$, we define
  \[
  g \colon C \to C,\; x\mapsto \left(1-\frac{\varepsilon}{R}\right) f(x) + \frac{\varepsilon}{R} x.
  \]
  From
  \[
  \|g(x)-g(y)\| \leq \left(1-\frac{\varepsilon}{R}\right) \|f(x)-f(y)\| 
  + \frac{\varepsilon}{R}\|x-y\|\leq \|x-y\|
  \]
  we conclude that $g$ is a non-expansive mapping and
  \[
  \|g(x)-f(x)\|= \frac{\varepsilon}{R} \|x-f(x)\| \leq \varepsilon
  \]
  shows $g\in B(f,\varepsilon)$. Using the fact that $f$ is constant on $U$, we get the identity $\|g(u)-g(v)\|=\frac{\varepsilon}{R}\|u-v\|$ and we deduce that all mappings $h\in B\left(g, \frac{\|u-v\|}{3R}\varepsilon\right)$ are non-constant on $U$.
\end{proof}

\begin{lemma}\label{theorem:local}
  $\N(U)$ is a $\sigma$-porous subset of $\M$.
\end{lemma}
\begin{proof}
	The family of all intervals $(a,b)\subseteq(0,1)$ satisfying~\eqref{eq:b-a} is an open cover of $(0,1)$. Since $(0,1)$ is a separable metric space and hence a Lindelöf space there exists a countable subcover $\{(a_i,b_i)\}_{i\in\mathbb{N}}$. Hence we may write
  \begin{equation*}
    \N(U)=\left(\bigcup_{i\in\mathbb{N}}\N_{a_i}^{b_i}(U)\right)\cup\N_{0}(U).
  \end{equation*}
  Applying Lemma~\ref{lemma:porous} and Lemma~\ref{lemma:constants} 
  we see that each of the sets in the above countable decomposition of $\N(U)$ is a porous set.
\end{proof}
We are now in a position to combine the results of the present section in proofs of 
our main results.
\begin{proof}[Proof of Theorem~\ref{theorem:mainresult}]
  Choosing $U=C$, the result follows from Lemma~\ref{theorem:local}.
\end{proof}

\begin{proof}[Proof of Theorem~\ref{theorem:mainresult2}]
  Since $X$ is separable, we may choose a countable dense subset $(x_{i})$ of $C$. Letting 
  $(q_{i})$ be an enumeration of $\Q\cap(0,1)$, we define for each pair $i,j\geq 1$ an open 
  subset $U_{i,j}=B(x_{i},q_{j})\cap C$ of $C$. Applying Lemma~\ref{theorem:local} with 
  $U=U_{i,j}$, we obtain that $\N(U_{i,j})$ is a $\sigma$-porous subset of $\M$ for each $i,j$. 
  We define the set $\widetilde{\N}$ by
  \begin{equation*}
    \widetilde{\N}=\bigcup_{i,j}\N(U_{i,j}).
  \end{equation*}
  Clearly $\widetilde{\N}$ is a $\sigma$-porous subset of $\M$. Fix a mapping $f\in\M\setminus \widetilde{\N}$. To 
  complete the proof, we need to verify that the set $R(f)$ defined in \eqref{eq:Rf} is a 
  residual subset of $C$. To this end, observe that 
  \begin{align*}
    R(f)&=\bigcap_{s\in\Q\cap(0,1)}\;\bigcap_{r\in\Q\cap(0,1)}\left\{x\in C\quad\colon\quad \lip(f,x,r)> s\right\},
  \end{align*}
  where we define $\lip(f,x,r)$ by
  \begin{equation*}
  \lip(f,x, r):= \sup \left\{\frac{\left\|f(y)-f(x)\right\|}{\left\|y-x\right\|}\quad\colon\quad y\in B(x,r)\cap C,\quad y\neq x\right\}.
  \end{equation*}
  Thus, it suffices to show that each of the sets expressed in the above intersection is an 
  open, dense subset of $C$. That these sets are open, is readily verified by checking that 
  their complements are closed. For fixed $r$ and $s$, consider a sequence $x_{n}$ with $\lip(f,x_{n},r)\leq s$, which converges to a point $x\in C$. We need to show that $\lip(f,x,r)\leq s$. Fixing $\varepsilon>0$ and $y\in B(x,r)\cap C$, we choose $x_{k}$ such that $\left\|x_{k}-x\right\|<\min\left\{r-\left\|y-x\right\|,\varepsilon/\lip(f)\right\}$. Then, using $\lip(f,x_{k},r)\leq s$, we obtain
  \begin{equation*}
  \left\|f(y)-f(x)\right\|\leq\left\|f(y)-f(x_{k})\right\|+\left\|f(x_{k})-f(x)\right\|\leq s\left\|y-x_{k}\right\|+\varepsilon.
  \end{equation*}
  Letting $\varepsilon\to 0$ and therefore $x_{k}\to x$ completes the argument.
  
  To establish density, we fix $r,s\in\Q\cap(0,1)$ and an open subset $U$ of $C$. Choose 
  $j,k\geq 1$ so that $U_{j,k}\subseteq U$. Since $f\notin\N(U_{j,k})$, it follows from Lemma~\ref{lemma:x0} that $\sup_{x\in U_{j,k}}\lip(f,x)=1$. This implies the existence of a point $y\in U_{j,k}\subseteq U$ such that $\lip(f,y)>s$ and therefore $\lip(f,y,r)>s$ for all $r>0$. Hence 
  \begin{equation*}
  U\cap\left\{x\in C\quad\colon\quad \lip(f,x,r)>s\right\}\neq \emptyset.
  \end{equation*}
\end{proof}

\begin{remark}
  For each $f\in\mathcal{N}$ the mapping 
  \[
  [0,\varepsilon_0)\to \mathcal{C}(C,C), \varepsilon \mapsto g_\varepsilon,
  \]
  where $g_\varepsilon$ is chosen as in \eqref{eq:geps}
  with $\left\|g_{\varepsilon}-f\right\|_{\infty}\leq \varepsilon$ and $B(g_\varepsilon, \alpha \varepsilon)\cap\N_{a}^{b}(U)=\emptyset$,
  is a Lipschitz curve. It would be of interest to check whether such a curve could 
  be differentiable. This would provide information about the directions from which
  the $g_\varepsilon$ approach the strict contraction $f$.
\end{remark}
\vspace{4mm}\noindent
\textbf{Acknowledgement.} The authors would like to thank Eva Kopeck\'{a} for fruitful discussions
and remarks, and Simeon Reich for helpful comments and suggestions. The authors also wish to thank the referee for improvements of the original proof.

\def\cprime{$'$}
\providecommand{\MR}{\relax\ifhmode\unskip\space\fi MR }
\providecommand{\MRhref}[2]{%
	\href{http://www.ams.org/mathscinet-getitem?mr=#1}{#2}
}
\providecommand{\href}[2]{#2}

\end{document}